\documentclass[english]{article}
\usepackage[T1]{fontenc}
\usepackage[latin9]{inputenc}
\usepackage{amsmath}
\usepackage{amsthm}

\makeatletter
\theoremstyle{plain}
\newtheorem{thm}{\protect\theoremname}
\theoremstyle{plain}
\newtheorem{cor}[thm]{\protect\corollaryname}
\theoremstyle{plain}
\newtheorem{lem}[thm]{\protect\lemmaname}
\ifx\proof\undefined
\newenvironment{proof}[1][\protect\proofname]{\par
	\normalfont\topsep6\p@\@plus6\p@\relax
	\trivlist
	\itemindent\parindent
	\item[\hskip\labelsep\scshape #1]\ignorespaces
}{%
	\endtrivlist\@endpefalse
}
\providecommand{\proofname}{Proof}
\fi
\newcommand{\lyxaddress}[1]{
	\par {\raggedright #1
	\vspace{1.4em}
	\noindent\par}
}

\@ifundefined{date}{}{\date{}}
\makeatother

\usepackage{babel}
\providecommand{\corollaryname}{Corollary}
\providecommand{\lemmaname}{Lemma}
\providecommand{\theoremname}{Theorem}

\begin{document}
\title{The Gauss map of hypersurfaces with constant weighted mean curvature
in the Gaussian space}
\author{Michael Gomez, Matheus Vieira}
\maketitle
\begin{abstract}
In this paper we study the Gauss map of hypersurfaces with constant
weighted mean curvature in the Gaussian space. We show that if the
image of the Gauss map is in a closed hemisphere, then the hypersurface
is a hyperplane or a generalized cylinder. We also show that if the
image of the Gauss map is in $S^{n}\setminus\bar{S}_{+}^{n-1}$, then
the hypersurface is a hyperplane. This generalizes previous results
for self-shrinkers obtained by Ding-Xin-Yang.
\end{abstract}

\section{Introduction}

One of the main problems in the theory of the mean curvature flow
is to understand self-shrinkers because they are singularity models
for the flow (see the survey \cite{CMP2015}). In this paper we are
interested in hypersurfaces with constant weighted mean curvature
in the Gaussian space, which generalize self-shrinkers.

Let $M$ be a hypersurface in the Gaussian space $R^{n+1}$ (the Euclidean
space with the Gaussian measure, for more details see Section 2).
We say that $M$ has constant weighted mean curvature $\lambda$ if
\[
H+\frac{1}{2}\langle x,N\rangle=\lambda.
\]
Here $H$ is the mean curvature of $M$ (using the convention that
$H$ is the trace of second fundamental form), $x$ is the position
vector of $R^{n+1}$, $N$ is the unit normal vector of $M$ and $\lambda$
is a constant. Hypersurfaces with constant weighted mean curvature
$\lambda$ in the Gaussian space $R^{n+1}$ (also known as $\lambda$-hypersurfaces)
are critical points of the weighted area functional for variations
preserving the enclosed weighted volume (see \cite{MR2015}, \cite{CW2018}).
Self-shrinkers in $R^{n+1}$ correspond to the special case of weighted
mean curvature $\lambda=0$.

There are many interesting papers about hypersurfaces with constant
weighted mean curvature, for example \cite{COW2016}, \cite{WX2017},
\cite{G2018}, \cite{WP2019}, \cite{AM2021}, \cite{CW2021}, \cite{S2021},
\cite{MV2022}. In this paper we are interested in the Gauss map of
these hypersurfaces. In \cite{BS2014}, Bao-Shi obtained results for
the Gauss map of translating solitons with bounded mean curvature.
In \cite{CMR2020}, Colombo-Mari-Rigoli obtained results for the Gauss
map of self-shrinkers and translating solitons in certain warped products.
In \cite{DXY2016}, Ding-Xin-Yang proved that for a complete self-shrinker
properly immersed in $R^{n+1}$, if the image of the Gauss map is
in a hemisphere, then the self-shrinker is a hyperplane or a generalized
cylinder (see also \cite{VZ2018}). They also proved that if the image
of the Gauss map is in $S^{n}\setminus\bar{S}_{+}^{n-1}$, then the
self-shrinker is a hyperplane. In this paper we generalize the results
of Ding-Xin-Yang to hypersurfaces with constant weighted mean curvature.

We would like to mention that by results of Cheng-Vieira-Zhou \cite{CVZ2021},
a complete hypersurface with constant weighted mean curvature in the
Gaussian space $R^{n+1}$ is properly immersed if and only if it has
finite weighted volume (see Theorem \ref{thm:finitevolume}).

For a complete hypersurface with constant weighted mean curvature
properly immersed in the Gaussian space $R^{n+1}$, we show that if
the image of the Gauss map is in a hemisphere, then the hypersurface
is a hyperplane or a generalized cylinder with a factor of constant
weighted mean curvature. This generalizes Theorem 3.1 in \cite{DXY2016}.
\begin{thm}
\label{thm:hemisphere}Let $M$ be a complete hypersurface with constant
weighted mean curvature $\lambda$ properly immersed in the Gaussian
space $R^{n+1}$. Suppose that the image of the Gauss map is in the
closed hemisphere $x_{n+1}\geq0$. Then either:

(i) $M$ is a hyperplane, or

(ii) $M$ splits isometrically as $M^{n}=\tilde{M}^{n-1}\times R$,
where $\tilde{M}^{n-1}$ is a hypersurface with constant weighted
mean curvature $\lambda$ in the Gaussian space $R^{n}\times\left\{ 0\right\} $.
\end{thm}
Applying this result to the factor, we get:
\begin{cor}
Let $M$ be a complete hypersurface with constant weighted mean curvature
$\lambda$ properly immersed in the Gaussian space $R^{n+1}$. Suppose
that the image of the Gauss map is in the intersection of the closed
hemispheres $x_{n+2-p}\geq0,\dots,x_{n+1}\geq0$ with $p\geq1$. Then
either:

(i) $M$ is a hyperplane, or

(ii) $M$ splits isometrically as $M^{n}=\tilde{M}^{n-p}\times R^{p}$,
where $\tilde{M}^{n-p}$ is a hypersurface with constant weighted
mean curvature $\lambda$ in the Gaussian space $R^{n+1-p}\times\left\{ 0\right\} ^{p}$.
\end{cor}
For a complete hypersurface with constant weighted mean curvature
properly immersed in the Gaussian space $R^{n+1}$, we show that if
the image of the Gauss map is in $S^{n}\setminus\bar{S}_{+}^{n-1}$,
then the hypersurface is a hyperplane. This generalizes Theorem 3.2
in \cite{DXY2016}.

Here $S^{n}\setminus\bar{S}_{+}^{n-1}$ is the set of all $x=\left(x_{1},\dots,x_{n+1}\right)$
in $S^{n}$ such that $\left(x_{1},x_{2}\right)$ is in $R^{2}\setminus\left([0,\infty)\times\left\{ 0\right\} \right)$.
For example, $S^{2}\setminus\bar{S}_{+}^{1}$ is $S^{2}$ minus the
closure of a semicircle.
\begin{thm}
\label{thm:semicircle}Let $M$ be a complete hypersurface with constant
weighted mean curvature $\lambda$ properly immersed in the Gaussian
space $R^{n+1}$. Suppose that the image of the Gauss map is in $S^{n}\setminus\bar{S}_{+}^{n-1}$.
Then $M$ is a hyperplane.
\end{thm}
We would like to mention that Theorem 1.1 and Theorem 1.2 are also
considered in \cite{CW2014preprint}, but proof is sketchy (the paper
is not published yet). Our proof is completely different and it is
in detail.

In Section 2 we introduce the definitions, notations and conventions
of this paper. In Section 3 we prove some lemmas. In Section 4 we
prove Theorem \ref{thm:hemisphere}. In Section 5 we prove Theorem
\ref{thm:semicircle}.

We would like to thank Detang Zhou for the support and guidance.

\section{Preliminaries}

In this section we introduce the definitions, notations and conventions
of this paper. Then we fix them in the special case of hypersurfaces
in the Gaussian space.

\paragraph{Smooth metric measure spaces}

We say that $\left(M,g,f\right)$ is a smooth metric measure space
if $(M,g)$ is a Riemannian manifold and $f$ is a function on $M$.
The weighted volume is defined by
\[
dvol_{f}=e^{-f}dvol,
\]
where $dvol$ is the Riemannian volume of $M$. If $u$ is a function
on $M$, the weighted Laplacian of $u$ is defined by
\[
\Delta_{f}u=\Delta u-\left\langle \nabla f,\nabla u\right\rangle ,
\]
where $\nabla$ is the Levi-Civita connection of $M$ and $\Delta$
is the Laplacian of $M$. If $X$ is a vector field on $M$, the weighted
divergence of $X$ is defined by
\[
div_{f}X=divX-\left\langle \nabla f,X\right\rangle ,
\]
where $div$ is the divergence operator of $M$. Note that
\[
\Delta_{f}=div_{f}\nabla.
\]
The weighted Laplacian is a self-adjoint operator with respect to
the weighted volume, that is, if $u$ and $v$ are functions on $M$
with compact support then
\[
\int_{M}u\left(\Delta_{f}v\right)e^{-f}=-\int_{M}\left\langle \nabla u,\nabla v\right\rangle e^{-f}.
\]
More generally, if $X$ is a vector field on $M$ with compact support
then
\[
\int_{M}\left(div_{f}X\right)e^{-f}=0.
\]

\paragraph{Weighted mean curvature}

Let $M^{n}$ be a submanifold in a smooth metric measure space $\left(\bar{M}^{n+k},\bar{g},f\right)$.
The second fundamental form of $M$ is defined by
\[
A\left(X,Y\right)=\left(\bar{\nabla}_{X}Y\right)^{\perp},
\]
where $\bar{\nabla}$ is the Levi-Civita connection of $\bar{M}$
and $\perp$ is the projection to $TM^{\perp}$. The weighted mean
curvature vector of $M$ is defined by
\[
\vec{H}_{f}=tr_{TM}A+\left(\bar{\nabla}f\right)^{\perp}.
\]
Now suppose that $M$ is a hypersurface. The (scalar) weighted mean
curvature of $M$ is defined by
\[
H_{f}=\left\langle \vec{H}_{f},N\right\rangle ,
\]
where $N$ is the unit normal vector of $M$.

Let $M^{n}$ be a hypersurface in a smooth metric measure space $\left(\bar{M}^{n+1},\bar{g},f\right)$.
We say that $M$ has constant weighted mean curvature $\lambda$ if
\[
H_{f}=\lambda,
\]
where $\lambda$ is a constant. We say that $M$ is $f$-minimal if
\[
H_{f}=0.
\]
Note that $f$-minimal hypersurfaces are special cases of constant
weighted mean curvature hypersurfaces.

\paragraph{Hypersurfaces in the Gaussian space}

Now we review the definitions, notations and conventions introduced
above in the special case of hypersurfaces in the Gaussian space.
The Gaussian space $R^{n+1}$ is defined as the smooth metric measure
space $\left(R^{n+1},\bar{g},f\right)$, where $\bar{g}$ is the canonical
metric of $R^{n+1}$ and $f\left(x\right)=\left|x\right|^{2}/4$.
Let $M^{n}$ be a hypersurface in the Gaussian space $R^{n+1}$. Then
$\left(M,g,f\right)$ is also a smooth metric measure space, where
$g$ is the metric induced by $\bar{g}$. The weighted volume of $M$
is given by
\[
dvol_{f}=e^{-\left|x\right|^{2}/4}dvol,
\]
where $dvol$ is the Riemannian volume of $M$. The weighted Laplacian
of $M$ is given by
\[
\Delta_{f}u=\Delta u-\frac{1}{2}\left\langle x,\nabla u\right\rangle ,
\]
where $\nabla$ is the Levi-Civita connection of $M$, $\Delta$ is
the Laplacian of $M$ and $x$ is the position vector of $R^{n+1}$.
We would like to mention that $\Delta_{f}=\mathcal{L}$ in the notation
of Colding and Minicozzi \cite{CM2012}. The weighted divergence of
$M$ is given by
\[
div_{f}V=divV-\frac{1}{2}\left\langle x,V\right\rangle ,
\]
where $div$ is the divergence operator of $M$. The weighted mean
curvature of $M$ is given by
\[
H_{f}=H+\frac{1}{2}\langle x,N\rangle,
\]
where $H$ is the mean curvature of $M$ and $N$ is the unit normal
vector of $M$. In particular, $M$ has constant weighted mean curvature
$\lambda$ if and only if
\[
H+\frac{1}{2}\langle x,N\rangle=\lambda.
\]

\section{Lemmas}

We use the following result of Cheng-Vieira-Zhou (Theorem 1.4 in \cite{CVZ2021})
in the proof of Theorem \ref{thm:hemisphere} and Theorem \ref{thm:semicircle}.
\begin{thm}
\label{thm:finitevolume}Let $M$ be a complete hypersurface with
constant weighted mean curvature $\lambda$ in the Gaussian space
$R^{n+1}$. Then the following conditions are equivalent to each other:

(i) $M$ is properly immersed.

(ii) $M$ has finite weighted volume, that is
\[
vol_{f}\left(M\right)=\int_{M}e^{-f}<\infty.
\]

(iii) $M$ has polynomial volume growth.
\end{thm}
We use the following result of Li-Wang (Proposition 1.1 in \cite{LW2006})
in the proof of Theorem \ref{thm:hemisphere}.
\begin{lem}
\label{lem:liwang}Let $\left(M,g,f\right)$ be a smooth metric measure
space and let $h$ be a function on $M$ such that $h\geq0$ (not
identically zero) and
\[
\Delta_{f}h+qh\leq0,
\]
where $q$ is a continuous function on $M$. Then
\[
\int_{M}q\phi^{2}e^{-f}\leq\int_{M}\left|\nabla\phi\right|^{2}e^{-f},
\]
for all $\phi$ in $C_{c}^{\infty}\left(M\right)$.
\end{lem}
We use the following result in the proof of Theorem \ref{thm:hemisphere}
and Lemma \ref{lem:laplacianthetanv}.
\begin{lem}
\label{lem:laplaciannv}Let $M$ be a hypersurface in the smooth metric
measure space $\left(R^{n+1},\bar{g},f\right)$, where $\bar{g}$
is the the canonical metric of $R^{n+1}$ and $f$ is a function on
$R^{n+1}$ such that $\bar{\nabla}\bar{\nabla}f=c\bar{g}$ for some
$c$ constant. Let $V$ be a fixed vector in $R^{n+1}$. Then
\[
\Delta_{f}\left\langle N,V\right\rangle =-\left|A\right|^{2}\left\langle N,V\right\rangle -\left\langle \nabla H_{f},V\right\rangle .
\]
\end{lem}
\begin{proof}
Let $p$ be a point in $M$. We will prove the above formula at this
point. Let $e_{1},\dots,e_{n}$ be a local orthonormal basis of $TM$
such that $\left(\nabla e_{i}\right)\left(p\right)=0$. We can identify
the second fundamental form with a $\left(0,2\right)$-tensor, that
is, we can assume that $A\left(X,Y\right)=\left\langle \bar{\nabla}_{X}Y,N\right\rangle $.
We write
\[
v=\left\langle N,V\right\rangle .
\]

Step 1. We have
\[
\Delta v=-\left|A\right|^{2}v-\left\langle \nabla H,V^{T}\right\rangle .
\]

Proof. We have
\begin{align*}
e_{i}v & =\left\langle \bar{\nabla}_{e_{i}}N,V\right\rangle +\left\langle N,\bar{\nabla}_{e_{i}}V\right\rangle \\
 & =\left\langle \bar{\nabla}_{e_{i}}N,V\right\rangle .
\end{align*}
Since $\bar{\nabla}_{e_{i}}N=-\sum_{j}A\left(e_{i},e_{j}\right)e_{j}$
we have
\[
e_{i}v=-\sum_{j}A\left(e_{i},e_{j}\right)\left\langle e_{j},V\right\rangle .
\]
Then
\begin{align*}
e_{i}e_{i}v & =-\sum_{j}\left\lbrace e_{i}\left(A\left(e_{i},e_{j}\right)\right)\left\langle e_{j},V\right\rangle +A\left(e_{i},e_{j}\right)\left(\left\langle \bar{\nabla}_{e_{i}}e_{j},V\right\rangle +\left\langle e_{j},\bar{\nabla}_{e_{i}}V\right\rangle \right)\right\rbrace \\
 & =-\sum_{j}\left\lbrace e_{i}\left(A\left(e_{i},e_{j}\right)\right)\left\langle e_{j},V\right\rangle +A\left(e_{i},e_{j}\right)\left\langle \bar{\nabla}_{e_{i}}e_{j},V\right\rangle \right\rbrace .
\end{align*}
Since $\left(\nabla e_{i}\right)\left(p\right)=0$ we have $e_{i}\left(A\left(e_{i},e_{j}\right)\right)=\left(\nabla_{e_{i}}A\right)\left(e_{i},e_{j}\right)$
and $\bar{\nabla}_{e_{i}}e_{j}=A\left(e_{i},e_{j}\right)N$. Then
\[
e_{i}e_{i}v=-\sum_{j}\left(\nabla_{e_{i}}A\right)\left(e_{i},e_{j}\right)\left\langle e_{j},V\right\rangle -\sum_{j}A\left(e_{i},e_{j}\right)^{2}\left\langle N,V\right\rangle .
\]
By the Codazzi equation we have $\left(\nabla_{e_{i}}A\right)\left(e_{i},e_{j}\right)=\left(\nabla_{e_{j}}A\right)\left(e_{i},e_{i}\right)$.
Then
\[
e_{i}e_{i}v=-\left(\nabla_{V^{T}}A\right)\left(e_{i},e_{i}\right)-\sum_{j}A\left(e_{i},e_{j}\right)^{2}\left\langle N,V\right\rangle .
\]
Summing in $i$ and using again the fact that $\left(\nabla e_{i}\right)\left(p\right)=0$
we get
\[
\Delta v=-\left|A\right|^{2}v-\left\langle \nabla H,V^{T}\right\rangle .
\]
The proof of Step 1 is complete.

Step 2. We have
\[
\nabla H_{f}=\nabla H-\sum_{i}A\left(\nabla f,e_{i}\right)e_{i}.
\]

Proof. Since $H_{f}=H+\left\langle \bar{\nabla}f,N\right\rangle $
we have 
\[
e_{i}H_{f}=e_{i}H+\left\langle \bar{\nabla}_{e_{i}}\bar{\nabla}f,N\right\rangle +\left\langle \bar{\nabla}f,\bar{\nabla}_{e_{i}}N\right\rangle .
\]
Since $\bar{\nabla}_{e_{i}}N=-\sum_{j}A\left(e_{i},e_{j}\right)e_{j}$
and $\bar{\nabla}\bar{\nabla}f=c\bar{g}$ we have 
\begin{align*}
e_{i}H_{f} & =e_{i}H+\bar{\nabla}\bar{\nabla}f\left(e_{i},N\right)-\sum_{j}A\left(e_{i},e_{j}\right)\left\langle \bar{\nabla}f,e_{j}\right\rangle \\
 & =e_{i}H-A\left(\nabla f,e_{i}\right).
\end{align*}
The proof of Step 2 is complete.

Finally, since $e_{i}v=-\sum_{j}A\left(e_{i},e_{j}\right)\left\langle e_{j},V\right\rangle $
we have
\begin{align*}
\Delta_{f}v & =\Delta v-\left\langle \nabla f,\nabla v\right\rangle \\
 & =\Delta v+A\left(\nabla f,V^{T}\right).
\end{align*}
By the previous steps we have
\begin{align*}
\Delta_{f}v & =-\left|A\right|^{2}v-\left\langle \nabla H,V^{T}\right\rangle +A\left(\nabla f,V^{T}\right)\\
 & =-\left|A\right|^{2}v-\left\langle \nabla H_{f},V^{T}\right\rangle \\
 & =-\left|A\right|^{2}v-\left\langle \nabla H_{f},V\right\rangle .
\end{align*}
The proof of the lemma is complete.
\end{proof}
We use the following lemma in the proof of Lemma \ref{lem:laplacianthetanv}.

Note that if $\left(u,v\right)$ is in $R^{2}\setminus\left([0,\infty)\times\left\{ 0\right\} \right)$
with $u\neq0$ and $v\neq0$ then
\[
\arctan\left(\frac{v}{u}\right)=\frac{\pi}{2}-\arctan\left(\frac{u}{v}\right).
\]

\begin{lem}
\label{lem:laplaciantheta}Let $\left(M,g,f\right)$ a smooth metric
measure space and let $u$ and $v$ be functions on $M$ such that
$\left(u\left(p\right),v\left(p\right)\right)$ is in $R^{2}\setminus\left([0,\infty)\times\left\{ 0\right\} \right)$
for all $p$ in $M$. Let $\theta$ be the function on $M$ given
by
\[
\theta\left(p\right)=\begin{cases}
\arctan\left(\frac{v\left(p\right)}{u\left(p\right)}\right) & \mathrm{if}\,\,\,u\left(p\right)\neq0,\\
\frac{\pi}{2}-\arctan\left(\frac{u\left(p\right)}{v\left(p\right)}\right) & \mathrm{if}\,\,\,v\left(p\right)\neq0,
\end{cases}
\]
and let $\rho$ be the function on $M$ given by
\[
\rho\left(p\right)=\sqrt{u\left(p\right)^{2}+v\left(p\right)^{2}}.
\]
Then
\[
div_{f}\left(\rho^{2}\nabla\theta\right)=u\Delta_{f}v-v\Delta_{f}u,
\]
and
\[
\rho\Delta_{f}\rho=u\Delta_{f}u+v\Delta_{f}v+\rho^{2}\left|\nabla\theta\right|^{2}.
\]
\end{lem}
\begin{proof}
Let $p$ be a point in $M$. We will prove the above formulas at this
point.

First we prove the formula of the weighted Laplacian of $\theta$.
We can assume that $u\left(p\right)\neq0$ because the other case
is similar. By continuity we have $u\neq0$ in an open set containing
$p$. Then
\[
\theta=\arctan\left(\frac{v}{u}\right)
\]
in an open set containing $p$. So
\begin{align*}
\nabla\theta & =\frac{1}{1+\left(\frac{v}{u}\right)^{2}}\nabla\left(\frac{v}{u}\right)\\
 & =\frac{1}{u^{2}+v^{2}}\left(u\nabla v-v\nabla u\right).
\end{align*}
Since $\rho^{2}=u^{2}+v^{2}$ we have
\[
\rho^{2}\nabla\theta=u\nabla v-v\nabla u.
\]
Then
\begin{align*}
div_{f}\left(\rho^{2}\nabla\theta\right) & =div_{f}\left(u\nabla v-v\nabla u\right)\\
 & =u\Delta_{f}v-v\Delta_{f}u.
\end{align*}

Now we prove the formula of the weighted Laplacian of $\rho$. Since
$\rho^{2}=u^{2}+v^{2}$ we have
\[
\rho\nabla\rho=u\nabla u+v\nabla v.
\]
Then
\[
div_{f}\left(\rho\nabla\rho\right)=div_{f}\left(u\nabla u+v\nabla v\right),
\]
that is
\[
\rho\Delta_{f}\rho+\left|\nabla\rho\right|^{2}=u\Delta_{f}u+v\Delta_{f}v+\left|\nabla u\right|^{2}+\left|\nabla v\right|^{2}.
\]
Since
\[
\left|\nabla\rho\right|^{2}=\frac{1}{\rho^{2}}\left(u^{2}\left|\nabla u\right|^{2}+v^{2}\left|\nabla v\right|^{2}+2uv\left\langle \nabla u,\nabla v\right\rangle \right),
\]
we have
\begin{align*}
\rho\Delta_{f}\rho & =u\Delta_{f}u+v\Delta_{f}v+\left(1-\frac{u^{2}}{\rho^{2}}\right)\left|\nabla u\right|^{2}+\left(1-\frac{v^{2}}{\rho^{2}}\right)\left|\nabla v\right|^{2}-2\frac{uv}{\rho^{2}}\left\langle \nabla u,\nabla v\right\rangle \\
 & =u\Delta_{f}u+v\Delta_{f}v+\frac{1}{\rho^{2}}\left(v^{2}\left|\nabla u\right|^{2}+u^{2}\left|\nabla v\right|^{2}-2uv\left\langle \nabla u,\nabla v\right\rangle \right)\\
 & =u\Delta_{f}u+v\Delta_{f}v+\frac{1}{\rho^{2}}\left|u\nabla v-v\nabla u\right|^{2}.
\end{align*}
Since $u\nabla v-v\nabla u=\rho^{2}\nabla\theta$ we have
\[
\rho\Delta_{f}\rho=u\Delta_{f}u+v\Delta_{f}v+\rho^{2}\left|\nabla\theta\right|^{2}.
\]

The proof of the lemma is complete.
\end{proof}
We use the following lemma in the proof of Theorem \ref{thm:semicircle}.
\begin{lem}
\label{lem:laplacianthetanv}Let $M$ be a hypersurface with constant
weighted mean curvature $\lambda$ in the smooth metric measure space
$\left(R^{n+1},\bar{g},f\right)$, where $\bar{g}$ is the canonical
metric of $R^{n+1}$ and $f$ is a function on $R^{n+1}$ such that
$\bar{\nabla}\bar{\nabla}f=c\bar{g}$ for some $c$ constant. Suppose
that $\left(\left\langle N\left(p\right),E_{1}\right\rangle ,\left\langle N\left(p\right),E_{2}\right\rangle \right)$
is in $R^{2}\setminus\left([0,\infty)\times\left\{ 0\right\} \right)$
for all $p$ in $M$. Let $\theta$ be the function on $M$ given
by
\[
\theta\left(p\right)=\begin{cases}
\arctan\left(\frac{\left\langle N\left(p\right),E_{2}\right\rangle }{\left\langle N\left(p\right),E_{1}\right\rangle }\right) & \mathrm{if}\,\,\,\left\langle N\left(p\right),E_{1}\right\rangle \neq0,\\
\frac{\pi}{2}-\arctan\left(\frac{\left\langle N\left(p\right),E_{1}\right\rangle }{\left\langle N\left(p\right),E_{2}\right\rangle }\right) & \mathrm{if}\,\,\,\left\langle N\left(p\right),E_{2}\right\rangle \neq0,
\end{cases}
\]
and let $\rho$ be the function on $M$ given by
\[
\rho\left(p\right)=\sqrt{\left\langle N\left(p\right),E_{1}\right\rangle ^{2}+\left\langle N\left(p\right),E_{2}\right\rangle ^{2}},
\]
where $E_{1},\dots,E_{n+1}$ is the canonical basis of $R^{n+1}$.
Then
\[
div_{f}\left(\rho^{2}\nabla\theta\right)=0.
\]
\end{lem}
\begin{proof}
By Lemma \ref{lem:laplaciantheta} we have
\[
div_{f}\left(\rho^{2}\nabla\theta\right)=\left\langle N,E_{1}\right\rangle \Delta_{f}\left(\left\langle N,E_{2}\right\rangle \right)-\left\langle N,E_{2}\right\rangle \Delta_{f}\left(\left\langle N,E_{1}\right\rangle \right).
\]
By Lemma \ref{lem:laplaciannv} and the fact that the weighted mean
curvature is constant we have
\[
div_{f}\left(\rho^{2}\nabla\theta\right)=0.
\]
The proof of the lemma is complete.
\end{proof}

\section{Proof of Theorem \ref{thm:hemisphere}}

In this section we prove Theorem \ref{thm:hemisphere}.

By Lemma \ref{lem:laplaciannv} and the fact that the weighted mean
curvature is constant we have
\[
\Delta_{f}\left\langle N,E_{n+1}\right\rangle +\left|A\right|^{2}\left\langle N,E_{n+1}\right\rangle =0,
\]
where $E_{1},\dots,E_{n+1}$ is the canonical basis of $R^{n+1}$.
Since the image of the Gauss map is in the closed hemisphere $x_{n+1}\geq0$
we have
\[
\left\langle N,E_{n+1}\right\rangle \geq0.
\]

First we assume that $\left\langle N,E_{n+1}\right\rangle $ is not
identically zero. Then by Lemma \ref{lem:liwang} we have
\[
\int_{M}\left|A\right|^{2}\phi^{2}e^{-f}\leq\int_{M}\left|\nabla\phi\right|^{2}e^{-f},
\]
for all $\phi$ in $C_{c}^{\infty}\left(M\right)$. Let $\phi$ be
the cutoff function on $M$ given by
\[
\phi=\begin{cases}
1 & \mathrm{in}\,\,\,B_{R},\\
R+1-r & \mathrm{in}\,\,\,B_{R+1}\setminus B_{R},\\
0 & \mathrm{in}\,\,\,M\setminus B_{R+1},
\end{cases}
\]
where $B_{R}=B\left(p_{0},R\right)$ is the intrinsic ball and $r\left(p\right)=dist_{M}\left(p,p_{0}\right)$
is the intrinsic distance function. Then
\[
\int_{B_{R}}\left|A\right|^{2}e^{-f}\leq\int_{M\setminus B_{R}}e^{-f}.
\]
Taking the limit as $R\to\infty$, using the monotone convergence
theorem and using the fact that $M$ has finite weighted volume (by
Theorem \ref{thm:finitevolume}) we get
\[
A=0.
\]
Then $M$ is a hyperplane.

Now we assume that $\left\langle N,E_{n+1}\right\rangle $ is identically
zero. Then for each $p$ in $M$ we have that $E_{n+1}$ is in $T_{p}M$.
Then for each $p$ in $M$ we have a straight line which is included
in $M$, passes through $p$, meets $R^{n}\times\left\{ 0\right\} $
orthogonally and extends infinitely to both ends (since $M$ is complete).
Then $M$ splits isometrically as $M^{n}=\tilde{M}^{n-1}\times R$,
where $\tilde{M}^{n-1}$ is the intersection of $M^{n}$ and $R^{n}\times\left\{ 0\right\} $.
We will show that $H_{f}^{\tilde{M}}=H_{f}^{M}$. Let $e_{1},\dots,e_{n}$
be a local orthonormal basis of $TM$ such that $e_{1}=E_{n+1}$ and
$e_{2},\dots,e_{n}$ is a local orthonormal basis of $T\tilde{M}$.
Since $\left\langle N,E_{n+1}\right\rangle =0$ we have that $N$
is also the unit normal vector of $\tilde{M}^{n-1}$ (with respect
to $R^{n}\times\left\{ 0\right\} $). Then
\begin{align*}
H_{f}^{M} & =H_{M}+\left\langle \bar{\nabla}f,N_{M}\right\rangle \\
 & =\left\langle \bar{\nabla}_{E_{n+1}}E_{n+1},N_{M}\right\rangle +\sum_{i=2}^{n}\left\langle \bar{\nabla}_{e_{i}}e_{i},N_{M}\right\rangle +\left\langle \bar{\nabla}f,N_{M}\right\rangle \\
 & =\sum_{i=2}^{n}\left\langle \bar{\nabla}_{e_{i}}e_{i},N_{\tilde{M}}\right\rangle +\left\langle \bar{\nabla}f,N_{\tilde{M}}\right\rangle \\
 & =H_{f}^{\tilde{M}}.
\end{align*}
Here we used the fact that $\bar{\nabla}_{E_{n+1}}E_{n+1}=0$.

The proof of Theorem \ref{thm:hemisphere} is complete.

\section{Proof of Theorem \ref{thm:semicircle}}

In this section we prove Theorem \ref{thm:semicircle}.

Let $\theta$ and $\rho$ be the functions on $M$ as in Lemma \ref{lem:laplacianthetanv}.
Since the image of the Gauss map is in $S^{n}\setminus\bar{S}_{+}^{n-1}$
we have that $\left(\left\langle N\left(p\right),E_{1}\right\rangle ,\left\langle N\left(p\right),E_{2}\right\rangle \right)$
is in $R^{2}\setminus\left([0,\infty)\times\left\{ 0\right\} \right)$
for all $p$ in $M$. By Lemma \ref{lem:laplacianthetanv} we have
\[
div_{f}\left(\rho^{2}\nabla\theta\right)=0.
\]
Let $\phi$ be the cutoff function on $M$ given by
\[
\phi=\begin{cases}
1 & \mathrm{in}\,\,\,B_{R},\\
R+1-r & \mathrm{in}\,\,\,B_{R+1}\setminus B_{R},\\
0 & \mathrm{in}\,\,\,M\setminus B_{R+1},
\end{cases}
\]
where $B_{R}=B\left(p_{0},R\right)$ is the intrinsic ball and $r\left(p\right)=dist_{M}\left(p,p_{0}\right)$
is the intrinsic distance function. Multiplying the above equation
by $\theta\phi^{2}$ and integrating by parts we get
\begin{align*}
0= & \int_{M}\theta\phi^{2}div_{f}\left(\rho^{2}\nabla\theta\right)e^{-f}\\
= & -\int_{M}\left|\nabla\theta\right|^{2}\phi^{2}\rho^{2}e^{-f}-2\int_{M}\theta\phi\left\langle \nabla\theta,\nabla\phi\right\rangle \rho^{2}e^{-f}.
\end{align*}
Since $2\left|\theta\phi\left\langle \nabla\theta,\nabla\phi\right\rangle \right|\leq\frac{1}{2}\left|\nabla\theta\right|^{2}\phi^{2}+2\theta^{2}\left|\nabla\phi\right|^{2}$
we have
\[
\int_{M}\left|\nabla\theta\right|^{2}\phi^{2}\rho^{2}e^{-f}\leq4\int_{M}\theta^{2}\left|\nabla\phi\right|^{2}\rho^{2}e^{-f}.
\]
Since $\theta$ and $\rho$ are bounded (see the definition of $\theta$
and $\rho$ in the statement of Lemma \ref{lem:laplacianthetanv})
we have
\[
\int_{B_{R}}\left|\nabla\theta\right|^{2}\rho^{2}e^{-f}\leq C\int_{M\setminus B_{R}}e^{-f}.
\]
Taking the limit as $R\to\infty$, using the monotone convergence
theorem and using the fact that $M$ has finite weighted volume (by
Theorem \ref{thm:finitevolume}) we get
\[
\left|\nabla\theta\right|^{2}\rho^{2}=0.
\]
Since $\left(\left\langle N\left(p\right),E_{1}\right\rangle ,\left\langle N\left(p\right),E_{2}\right\rangle \right)$
is in $R^{2}\setminus\left([0,\infty)\times\left\{ 0\right\} \right)$
for all $p$ in $M$, we have $\rho>0$. Then
\[
\theta=\theta_{0},
\]
where $\theta_{0}$ is a constant. We have $\left\langle N\left(p\right),E_{1}\right\rangle =\rho\left(p\right)\cos\theta_{0}$
and $\left\langle N\left(p\right),E_{2}\right\rangle =\rho\left(p\right)\sin\theta_{0}$
(see the definition of $\theta$ and $\rho$ in the statement of Lemma
\ref{lem:laplacianthetanv}). Let $V=\left(\cos\theta_{0},\sin\theta_{0},0,\dots,0\right)$
be a fixed vector in $R^{n+1}$. We have
\begin{align*}
\left\langle N\left(p\right),V\right\rangle  & =\rho\left(p\right)\left(\left(\cos\theta_{0}\right)^{2}+\left(\sin\theta_{0}\right)^{2}\right)\\
 & =\rho\left(p\right)\\
 & >0.
\end{align*}
Then the image of the Gauss map is in the open hemisphere
\[
\cos\theta_{0}x_{1}+\sin\theta_{0}x_{2}>0.
\]
By the proof of Theorem \ref{thm:hemisphere} we have that $M$ is
a hyperplane.

The proof of Theorem \ref{thm:semicircle} is complete.

\lyxaddress{Michael Gomez, Instituto de Matemática e Estatística, Universidade
Federal Fluminense, Niterói, RJ, Brazil.\\
Email address: michael\_eddy@id.uff.br}

\lyxaddress{Matheus Vieira, Departamento de Matemática, Universidade Federal
do Espírito Santo, Vitória, ES, Brazil.\\
Email address: matheus.vieira@ufes.br}
\end{document}